\documentclass[12pt]{article}
\usepackage{latexsym, amsmath, amsfonts, amscd, amssymb, verbatim, amsxtra,amsthm}
\usepackage{pgf,tikz,pgfplots}
\usepackage{tkz-euclide}
\usepackage{graphicx}
\usepackage{color}
\usepackage{hyperref}
\usepackage[all]{xy}
\usepackage{mathrsfs}
\usepackage{tikz,tkz-tab}
\usepackage{parallel}
\usepackage{array}
\usepackage{xcolor}
\usepackage{verbatim}
\usepackage{tikz-3dplot}
\usepackage{comment}
\usepackage{tikz-cd}

\theoremstyle{plain}
\newtheorem{theorem}{Theorem}[section]
\newtheorem{proposition}[theorem]{Proposition}
\newtheorem{lemma}[theorem]{Lemma}
\newtheorem{corollary}[theorem]{Corollary}

\theoremstyle{definition}

\theoremstyle{remark}
\newtheorem{remark}[theorem]{Remark}

\title{\large{\bf The Joint Range of Quadratic Mapping on Hilbert Space}\thanks{This work was sponsored by Vietnam Ministry of Education and Training grant number: B2023-TDV-03. \, \, \, \, \, \, \, \, \, \, \, \, \, \, \, \, \, \, \, \, \, \, \, \, \, \, \, \, \, \, \, \, \, \, \, \, \, \, \, \, \, \, \, \, \, \, \, \, \, \, \, \, \, \, \, \, \, \, \, \, \, \, \, \, \, \, \, \, \,  \, \, \, \, \, {\bf{2010 \emph{Mathematics Subject Classification.}}} 90C20,  90C22, 90C32 \, \, \, \, \, \, \, \, \, \, \, \, \, \, \, \, \, \, \, \, \, \, \, \, \, \, \, \, \, \, \, \, \, \, \, {\bf{\emph{Key words and phrases.}}} Joint range of quadratic mapping, Convex quadratic programming, S-lemma, Hidden convex structure}}
\author{\small{Huu-Quang Nguyen\thanks{Corresponding author.}}}

\date{}

  \makeatletter
  \renewcommand\section{\@startsection {section}{1}{\z@}%
  	{-3.5ex \@plus -1ex \@minus -.2ex}%
  	{2.3ex \@plus.2ex}%
  	{\normalfont\large}}
  \renewcommand\subsection{\@startsection{subsection}{2}{\z@}%
  	{-3.25ex\@plus -1ex \@minus -.2ex}%
  	{1.5ex \@plus .2ex}%
  	{\normalfont\Large\bfseries}}
  \renewcommand\subsubsection{\@startsection{subsubsection}{3}{\z@}%
  	{-3.25ex\@plus -1ex \@minus -.2ex}%
  	{1.5ex \@plus .2ex}%
  	{\normalfont\large\bfseries}}

\begin{document}

\maketitle

\fontsize{12pt}{13.65pt}\selectfont

\begin{abstract}
	We present a novel technical method for analyzing the hidden convex structure embedded in the joint range of a quadratic mapping defined on a Hilbert space. Our approach stands out by relying exclusively on elementary mathematical principles.
\end{abstract}


 
\begin{center}  \section{Introduction}\end{center}
The study of convex structures hidden within the joint range of a quadratic map on $\mathbb{R}^n$ has attracted significant attention in the mathematical literature (see \cite{Dines}, \cite{Polyak}, \cite{Beck07}, \cite{TT} \cite{Flores-Carcamo}, \cite{Jeyakumar3}, \cite{ Flores-Opazo}, \cite{Flores-Opazo2}, \cite{Quang}). This topic has proven to be a fundamental tool in analyzing the strong duality of nonconvex quadratic programming. However, this problem is inherently difficult due to the complex geometric and algebraic properties of quadratic mappings.

In this paper, we explore the image of a quadratic mapping \( F \), which consists of two functions, \( f(x) \) and \( g(x) \), on a Hilbert space \( \mathcal{H} \). While previous works, such as references \cite{Baccari} and \cite{Contino}, have analyzed the convexity of the set \( F(\mathcal{H}) \), our focus shifts to a more practical problem: examining the convexity of the set \( F(\mathcal{H}) + \bigwedge \), where \( \bigwedge \) is a convex cone in \( \mathbb{R}^2 \). More specifically, Baccari in \cite{Baccari} and Contino in \cite{Contino} prove the following results:
%
\begin{theorem}[\cite{Baccari}]\label{dl1.1}
	Let $A_i: \mathcal{H} \rightarrow \mathcal{H}$ is self-adjoint for all $i=1,2$ and $\exists \mu=\left(\mu_1, \mu_2\right) \in \mathbb{R}^2 \mid \mu_1 A_1+\mu_2 A_2>0$. Assume that the set $F(\mathcal{H})=\left\{\left(f_1(x), f_2(x)\right) \in \mathbb{R}^2: x \in \mathcal{H}\right\}$ is closed, where $f_j(x)=\left\langle A_j x, x\right\rangle+\left\langle x, a_j\right\rangle+b_j, j=1,2$. 	
	Then, $F(\mathcal{H})=\left\{\left(f_1(x), f_2(x)\right) \in \mathbb{R}^2: x \in \mathcal{H}\right\}$ is a convex set.
\end{theorem}
\begin{theorem}[\cite{Contino}]\label{dl1.2}
	Let $\mathcal{H}$ be a real inner product space, $3 \leq \operatorname{dim}(\mathcal{H}) \leq \infty$. Let $A_1, A_2 \in L(\mathcal{H})$ be such that $\mu_1 A_1+\mu_2 A_2>0$ for some $\mu_1, \mu_2 \in \mathbb{R}, a_1, a_2 \in \mathcal{H}$ and $b_1, b_2 \in \mathbb{R}$. Let $F=\left(f_1, f_2\right)$ be the non-homogeneous quadratic form defined by $f_j(x)=\left\langle A_j x, x\right\rangle+\left\langle x, a_j\right\rangle+b_j, j=1,2$. Then
	$
	F(\mathcal{H})=\left\{\left(f_1(x), f_2(x)\right) \in \mathbb{R}^2: x \in \mathcal{H}\right\}
	$
	is convex.
\end{theorem}
Both of the above theorems are used to establish an important result known as a type of S-lemma on Hilbert space, which ensures that under the assumptions stated in Theorem \ref{dl1.1} or Theorem \ref{dl1.2}, if $f_2(x^*)<0$ for some $x^*$ in $\mathcal{H}$, the following statements are equivalent.
\begin{itemize}
	\item[(i)] $\not\exists x\in \mathcal{H}$ such that $f_1(x)<0$ and $f_2(x)\leq 0$.
	\item[(ii)] $\exists \lambda \geq 0$ such that $f_1(x)+\lambda f_2(x)\ge 0$ for all $x\in\mathcal{H}$.
\end{itemize}

Our main result ensures the convexity of the set \( F(\mathcal{H}) + \bigwedge \) without imposing any assumptions. Consequently, we derive a variant of the S-lemma as stated above, also without any assumptions. Notably, our method is independent of the dimension of \( \mathcal{H} \), making it applicable to both finite- and infinite-dimensional settings. Our results hold even when \( \mathcal{H} \) is replaced by an affine manifold in any Hilbert space. This work unifies and extends many key findings from the literature. Importantly, our approach is based solely on elementary mathematical principles.

 {\centering \section{Preliminaries}    }
In this section, we present some simple geometric properties of a parabola in \( \mathbb{R}^2 \), which can be found in any standard linear geometry book. Clearly, a parabolic curve in \( \mathbb{R}^2 \) can be described as the solution set of an equation of the form \( x^T A x + a^T x + a_0 = 0 \), where \( A \) is a symmetric \( 2 \times 2 \) matrix satisfying \(\text{rank}(A) = 1\) and \( A \succeq 0 \) (or \( A \preceq 0 \)), with \( a \not\in \mathcal{N}(A) = \{ x \in \mathbb{R}^2 : Ax = 0 \} \) and \( a_0 \in \mathbb{R} \). More generally, we cite a result from \cite{Audin} as follows.

\begin{proposition}
Let $\mathcal{C}$ be a proper conic with nonempty image in $\mathbb{R}^2$. There exists an aﬃne frame in which an equation of $\mathcal{C}$ has one (and
only one) of the forms
$$x^2+ y^2=1 ~ (\text{ellipse}), x^2- y^2=1~ (\text{hyperbola}) \text{ or } y^2- x=0 ~(\text{parabola}).$$
\end{proposition}

The remaining results can be left as exercises, but we also provide their proofs to maintain the completeness of the paper.
\begin{lemma}\label{kq1}
Given a parabola $\mathcal{C}$ in $\mathbb{R}^2$, let $\psi(x)=x^TAx+a^Tx+a_0=0$ be the equation of  $\mathcal{C}$, where $A\succeq 0$. Assume that $\bar{x}$ and $\bar{y}$ are two distinct points on $\mathcal{C}$, and let $z^*\in (\bar{x}, \bar{y})$. Then $\psi(z^*)<0$.
\end{lemma}

\begin{proof}
	It is easy to see that $\{\bar{x}+t(\bar{y}-\bar{x}): t\in \mathbb{R}\}$ represents the straight line ${d}$ that passes through $\bar{x}$ and $\bar{y}$, and that $z^*=\bar{x}+t^*(\bar{y}-\bar{x})$ for some $t^*\in (0, 1)$. Since $d\cap\mathcal{C}=\{\bar{x}, \bar{y}\}$, it follows that the quadratic equation in $t$, defined as $e(t):= \psi(\bar{x}+t(\bar{y}-\bar{x}))=0$, has two roots: $0$ and $1$. 
	Therefore $e(t)$ and the quadratic coefficient of $e(t)$ have opposite signs for any $t\in(0, 1)$. By the form of $\psi(x)$, the quadratic coefficient of $e(t)$ is $(\bar{y}-\bar{x})^TA(\bar{y}-\bar{x})$. Moreover, since $A\succeq 0$, the quadratic coefficient of $e(t)$ must be positive. Therefore, $e(t)<0$ for any $t$ between $0$ and $1$. Consequently,  $e(t^*)=\psi(z^*)<0$.
\end{proof}
The following result states that a parabola in $\mathbb{R}^2$ cannot enclose two rays formed by two linearly independent vectors.
\begin{lemma}\label{kq1a}
	Given a parabola $\mathcal{C}$ in $\mathbb{R}^2$, let $\psi(x)=x^TAx+a^Tx+a_0=0$ be the equation of  $\mathcal{C}$, where $A\succeq 0$. Assume that $\Gamma_1=\{z+t b: t\in (-\infty, 0]\}$ and $\Gamma_2=\{z+t c: t\in (-\infty, 0]\}$, where $\{b, c\}$ is linearly independent, and that  $\psi(z)<0$. Then, either $\Gamma_1\cap \mathcal{C}\ne \emptyset$ or   $\Gamma_2\cap \mathcal{C}\ne \emptyset$.
\end{lemma}
\begin{proof}
	Suppose, by contradiction, that $(\mathcal{C}\cap\Gamma_1) \cup (\mathcal{C}\cap\Gamma_2)$ is an empty set. This implies that the two equations of the form $\alpha t^2+\beta t+\gamma=0$,  
\begin{equation}\label{eq1}
\psi(z+tb)=((z+tb)^TA(z+tb)+a^T(z+tb)+a_0)=0 
\end{equation}
and
\begin{equation}\label{eq2}
	\psi(z+tc)=((z+tc)^TA(z+tc)+a^T(z+tc)+a_0)=0,
\end{equation}
have no solution in $(-\infty, 0]$. Now, if $b^TAb\ne 0$, then since $A\succeq 0$, we have $b^TAb>0$. Consequently, $\psi(z+tc)>0$ for sufficiently large $|t|$. This, combine with the condition $\psi(z)<0$, implies that the equation \eqref{eq1} has the solution in $(-\infty, 0]$,  leading to a contradiction. Hence  $b^TAb$ must be zero. By the same argument, $c^TAc$ must also be zero. 

Since $A$ is a symmetric matrix of order 2, $A\succeq 0$ and $b^TAb=c^TAc=0$ for two linearly independent vectors $\{b, c\}$, it follows that $A$ is the zero matrix. This  contradicts the assumption that $\psi(x)=x^TAx+a^Tx+a_0=0$ represents the equation of a parabola.
\end{proof} 

\begin{center}    { \section{The main result }   }\end{center}

We begin this section with some concepts and notations necessary for this paper. Let \(\mathcal{H}\) be a Hilbert space with inner product \(\langle \cdot, \cdot \rangle\). Define the quadratic functions \( f(x) \) and \( g(x) \) on \(\mathcal{H}\) by  
\[
f(x) = \langle Px, x \rangle + \langle p, x \rangle + p_0, \quad g(x) = \langle Qx, x \rangle + \langle q, x \rangle + q_0,
\]
where \( P \) and \( Q \) are linear self-adjoint operators on \(\mathcal{H}\), $p$ and $q$ are elements of $\mathcal{H}$, $p_0$ and $q_0$ are real numbers.

Let $b$ and $c$ be two linearly independent vectors in $\mathbb{R}^2$. The set  \begin{equation} 
	\bigwedge=\{\lambda b+\beta c :\lambda\geq 0, \beta\geq 0\}\nonumber
\end{equation}  forms a convex cone in $\mathbb{R}^2$. We will demonstrate that $F(\mathcal{H})+\bigwedge$ is convex.

\begin{lemma}\label{kq1b}
	Let $F(x)=(\langle Px, x\rangle+\langle p, x\rangle+ p_0, \langle Qx, x\rangle+\langle q, x\rangle+q_0)^T$ be a quadratic map from $\mathcal{H}$ to $\mathbb{R}^2$, and let $\curlyvee$ be a straight line in $\mathcal{H}$. Then $F(\curlyvee)$ is either a point, a ray, a straight line or a parabola in $\mathbb{R}^2$.
\end{lemma}
\begin{proof}
	Assume that $\{\bar{x}+t(\bar{y}-\bar{x}): t\in \mathbb{R}\}$ represents the straight line $\curlyvee$, where $\bar{x}$ and $\bar{y}$ are two given vectors in $\mathcal{H}$. Then 
	\begin{equation}
		F(\curlyvee)=\{(\alpha t^2+\beta t+ \gamma, \alpha' t^2+\beta' t+ \gamma')^T: t\in\mathbb{R}\},
	\end{equation}where $$\alpha t^2+\beta t+ \gamma=\langle P(\bar{x}+t(\bar{y}-\bar{x})), \bar{x}+t(\bar{y}-\bar{x})\rangle+\langle p, \bar{x}+t(\bar{y}-\bar{x})\rangle+p_0, $$ and $$\alpha' t^2+\beta' t+ \gamma'=\langle Q(\bar{x}+t(\bar{y}-\bar{x})), \bar{x}+t(\bar{y}-\bar{x})\rangle+\langle q, \bar{x}+t(\bar{y}-\bar{x})\rangle+q_0.$$
	
	\noindent{\bf Case 1:} $\begin{vmatrix} \alpha &\beta\\	\alpha'&\beta'\end{vmatrix}=0$. 
	
	- If $\alpha=\beta=\alpha'=\beta'=0$ then $F(\curlyvee)$ is a point. 
	
	- If $\alpha^2+\beta^2+\alpha'^2+\beta'^2\ne 0$, then, without loss of generality, we assume that $(\alpha, \beta)\ne (0, 0)$.  Hence, $(\alpha', \beta')=k(\alpha, \beta)$ for some $k\in\mathbb{R}$. Considering an affine transformation $L:\mathbb{R}^2\rightarrow\mathbb{R}^2: (u_1, u_2)^T\mapsto (u_1, -ku_1+u_2)^T$. We have \begin{equation}
		L(F(\curlyvee))=\{(\alpha t^2+\beta t+ \gamma, \gamma'-k\gamma)^T: t\in\mathbb{R}\}.
	\end{equation}
	Obviously, if $\alpha=0$ then $\beta$ must be nonzero, and consequently, $L(F(\curlyvee))$ is a straight line. If $\alpha\ne 0$, then $\alpha t^2+\beta t+ \gamma$ is either bounded from below or bounded from above, and consequently, $L(F(\curlyvee))$ is a ray. Thus, $L(F(\curlyvee))$ is either a line or a ray, and the same holds for $F(\curlyvee)=L^{-1}(L(F(\curlyvee)))$. 
	
	\noindent{\bf Case 2:} $\begin{vmatrix} \alpha &\beta\\	\alpha'&\beta'\end{vmatrix}\ne 0$. Obviously, $(\alpha, \alpha')\ne (0, 0)$. Without loss of generality, we assume that $\alpha\ne 0$. Considering an affine transformation 
	$\bar{L}:\mathbb{R}^2\rightarrow\mathbb{R}^2: (u_1, u_2)^T\mapsto (u_1, -\dfrac{\alpha'}{\alpha}u_1+u_2)^T$. We have 
	\begin{equation}\label{ct-05}
		\bar{L}(F(\curlyvee))=\{(\alpha t^2+\beta t+ \gamma, (\beta'-\dfrac{\alpha'}{\alpha}\beta)t+ \gamma'-\dfrac{\alpha'}{\alpha}\gamma)^T: t\in\mathbb{R}\}.
	\end{equation}
	
	Since $\begin{vmatrix} \alpha &\beta\\	\alpha'&\beta'\end{vmatrix}\ne 0$, we have $\beta'-\dfrac{\alpha'}{\alpha}\beta\ne 0$. Therefore, \eqref{ct-05} represents a parabola  in $\mathbb{R}^2$, and the same holds for $F(\curlyvee)=\bar{L}^{-1}(\bar{L}(F(\curlyvee)))$. 
\end{proof}

\begin{theorem}\label{kq2}
	Let $F(x)=(\langle Px, x\rangle+\langle p, x\rangle+ p_0, \langle Qx, x\rangle+\langle q, x\rangle+q_0)^T$ be a quadratic map from $\mathcal{H}$ to $\mathbb{R}^2$. Then the set $F(\mathcal{H})+\bigwedge$ is convex. 
\end{theorem}

\begin{proof}
	Assume that $u\in F(\mathcal{H})+\bigwedge$, $v\in F(\mathcal{H})+\bigwedge$, and $w\in (u, v)$, we need to prove $w$ also belongs to $F(\mathcal{H})+\bigwedge$.

	\begin{figure}[htb]
		\begin{center}
			\includegraphics[scale=0.7]{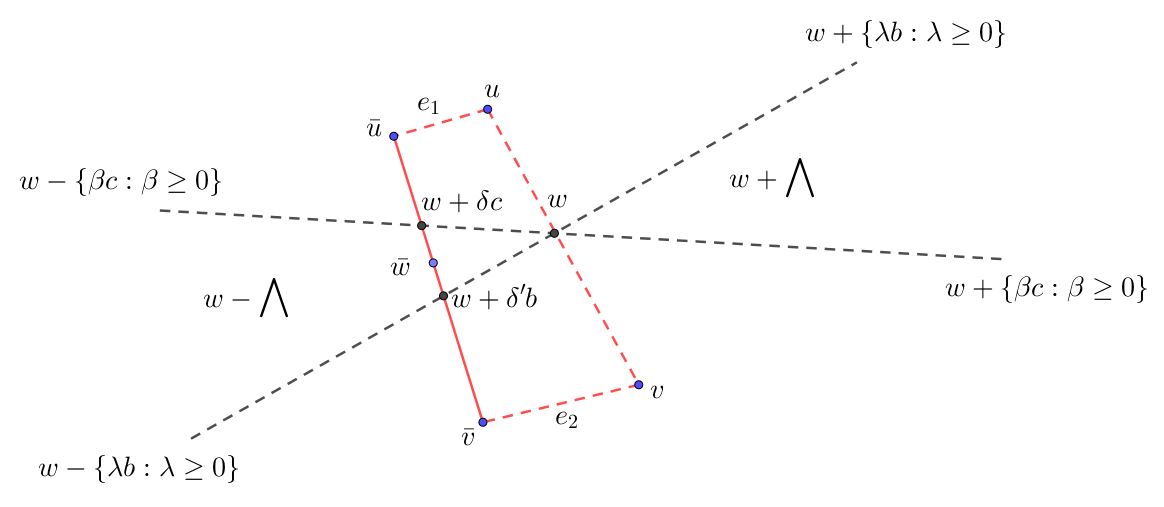}
		\end{center}
		\caption{}\label{pic1}
	\end{figure}
	
	It follows from $u, v\in F(\mathcal{H})+\bigwedge$ that there exist $\bar{u}, \bar{v}\in F(\mathcal{H})$, $e_1, e_2\in \bigwedge$ such that 
	\begin{equation}\label{ct1}
		u=\bar{u}+e_1, v=\bar{v}+e_2.
	\end{equation}

	Since $w\in (u, v)$, it follows that $w=\alpha u+\beta v$, where $\alpha, \beta\in (0, +\infty)$ and $\alpha+\beta=1.$ By \eqref{ct1}, $\alpha u=\alpha \bar{u}+\alpha e_1, \beta v=\beta\bar{v}+\beta e_2$. We have  
	\begin{eqnarray}
		w&=&\alpha u+\beta v \nonumber \\
		&=&\alpha \bar{u}+\alpha e_1 +\beta\bar{v}+\beta e_2 \nonumber\\
		&=&(\alpha \bar{u}+\beta\bar{v})+(\alpha e_1+\beta e_2) \nonumber\\
		&=&\bar{w}+(\alpha e_1+\beta e_2), \label{ct2} 
	\end{eqnarray}
	where  $\bar{w}=\alpha \bar{u}+\beta\bar{v}$. By the characteristics of $\alpha$ and $\beta$, as well as $e_1$ and $e_2$, we have 
	\begin{equation}\label{ct3}
		\bar{w}\in (\bar{u}, \bar{v}) \text{ and } \bar{w}\in w-\bigwedge.
	\end{equation}
	Then, in \eqref{ct3}, since $\bar{w}\in w-\bigwedge$, it follows that
	\begin{equation}\label{ct09aa}
		\begin{aligned}
			&\bar{w}&=&\quad w-\lambda b-\gamma c, \text{ for some } \lambda\geq 0 \text{ and } \gamma\geq 0.
		\end{aligned}
	\end{equation}
	\noindent{\bf Case 1:} If either $\bar{u}\in w-\bigwedge$ or $\bar{v}\in w-\bigwedge$. Then of course, either $w\in \bar{u}+\bigwedge$ or $w\in \bar{v}+\bigwedge$. This means that $w\in F(\mathcal{H})+\bigwedge.$
	
	\noindent{\bf Case 2:} Both $\bar{u}$ and $\bar{v}$ do not belong to $w-\bigwedge$. We  express ${\bar{u}-\bar{w}}$  in terms of the bases $\{b, c\}$ and assume that it is given by $\bar{u}-\bar{w}=sb+tc$, or equivalently, $\bar{u}=\bar{w}+sb+tc$. Since $\bar{w}\in (\bar{u}, \bar{v})$, there exists $l<0$ such that $\bar{v}=\bar{w}+l(sb+tc)$.
	
	Since $\bar{u}\not\in w-\bigwedge$ and $\bar{v}\not\in w-\bigwedge$, and by \eqref{ct3}, it follows that neither $(t\leq 0, s\leq 0)$ nor $(lt\leq 0, ls\leq 0)$ holds, where $l<0$. Thus $ts$ must be negative. Thus, without loss of generality, we assume that $s>0$ and  $t<0$;  otherwise, we swap the roles of $\bar{u}$ and $\bar{v}$. 
	
	Now, representing \( e_1 \) and \( e_2 \) in terms of the basis \( \{b, c\} \) (i.e., \( e_1 = \mu_1 b + \nu_1 c \), \( e_2 = \mu_2 b + \nu_2 c \)), and using \eqref{ct2}, we obtain 
	\begin{equation}\label{ct10aa}
		\begin{aligned}
			&\bar{u}&=&\quad \bar{w}+sb+tc\\
			& &=&\quad [w-\alpha e_1-\beta e_2]+sb+tc\\
			& &=&\quad [w-\alpha(\mu_1b+\nu_1c)-\beta(\mu_2b+\nu_2c)]+sb+tc\\
			& &=&\quad w+ (s-\alpha\mu_1-\beta\mu_2)b+(t-\alpha\nu_1-\beta\nu_2)c.\\
		\end{aligned}
	\end{equation}
	It follows from $e_1\in\bigwedge$, $e_2\in\bigwedge$ that $\mu_1, \nu_1,  \mu_2$ and $\nu_2$ are not negative. Note that in this case, $s>0$, $t<0$, $\alpha> 0$ and $\beta> 0$, we have 
	\begin{equation}\label{n01}
		(t-\alpha\nu_1-\beta\nu_2)<0. 
	\end{equation} 
	Therefore, since $\bar{u} =w+ (s-\alpha\mu_1-\beta\mu_2)b+(t-\alpha\nu_1-\beta\nu_2)c\not\in w-\bigwedge$, $(s-\alpha\mu_1-\beta\mu_2)$ must be positive. Consequently, $\lambda+(s-\alpha\mu_1-\beta\mu_2)>0$. Let
	\begin{equation}
		\delta_1=\dfrac{\lambda}{\lambda+(s-\alpha\mu_1-\beta\mu_2)}; ~\delta_2= \dfrac{s-\alpha\mu_1-\beta\mu_2}{\lambda+(s-\alpha\mu_1-\beta\mu_2)}.
\end{equation} 
	Then, it is easy that \begin{equation}\label{n00}
		\delta_1\geq 0, \delta_2\geq 0 \text{  and } \delta_1+\delta_2=1.
	\end{equation} Multiplying the vector $\bar{u}$ in \eqref{ct10aa}  by $\delta_1$, the vector  $\bar{w}$ in \eqref{ct09aa} by $\delta_2$, and then summing the results, we obtain:
	\begin{equation}\label{ct12aa}
		\delta_1\bar{u}+\delta_2\bar{w}=w+(\delta_1(t-\alpha\nu_1-\beta\nu_2)-\delta_2\gamma)c.
	\end{equation}
	Based on the characteristics of $\delta_1$, $\delta_2$, $t-\alpha\nu_1-\beta\nu_2$ and $\gamma$ as shown above (see  \eqref{ct09aa}, \eqref{n01} and \eqref{n00}), we obtain that $\delta_1\bar{u}+\delta_2\bar{w}\in[\bar{u}, \bar{w}]$ and $\delta:=(\delta_1(t-\alpha\nu_1-\beta\nu_2)-\delta_2\gamma)\leq 0.$ Hence, we have proved that
	\begin{equation}\label{ct9}
		\exists \delta\leq 0 \text{ such that } w+\delta c\in [\bar{u}, \bar{w}].
	\end{equation}

Take $\bar{x}, \bar{y}\in \mathcal{H}$ such that $F(\bar{x})=\bar{u}$ and $F(\bar{y})=\bar{v}$. Let $\curlyvee$ be the straight line that passes through $\bar{x}$ and $\bar{y}$. By Lemma \ref{kq1b}, $F(\curlyvee)$ is either a ray, a straight line or a parabola in $\mathbb{R}^2$.
	
	+) If $F(\curlyvee)$ is a ray or a straight line. Since $F(\curlyvee)$ passes though $\bar{u}$ and $\bar{v}$, so, it also passes though $\bar{w}$. It means that $\bar{w}\in F(\curlyvee)\subset F(\mathcal{H})$. Combining this with \eqref{ct09aa}, we obtain $w-\lambda b -\gamma c\in F(\mathcal{H})$ where $\lambda\geq 0$ and $\gamma\geq 0$. Therefore, $w\in F(\mathcal{H})+\bigwedge$.
	
	+) If $F(\curlyvee)$ is a parabola $\mathcal{C}$. Let $\psi(x)=x^TAx+a^Tx+a_0=0$ be the equation of parabola $\mathcal{C}$ in $\mathbb{R}^2$, where $A\succeq 0$. In this case, $\bar{u}\in \mathcal{C}$ and $\bar{v}\in \mathcal{C}$. By Lemma \ref{kq1}, we have  \begin{equation}\label{ct11}
		\psi(w+\delta c)\leq 0,
	\end{equation} where $w+\delta c$ is in \eqref{ct9}, 
	see Figure \ref{pic1}.
	
	If $\psi(w)\geq 0$, then by the Intermediate Value Theorem, there exists $w^*\in [w+\delta c, w]$ such that $\psi(w^*)=0.$ Since $\delta\leq 0$ and $w^*\in [w+\delta c, w]$, we have $w^*\in w-\bigwedge$. Moreover, $\psi(w^*)=0$ implies that $w^*\in F(\curlyvee)\subset F(\mathcal{H})$. Hence, we obtain that $w\in w^*+\bigwedge\subset F(\mathcal{H})+\bigwedge$.
	
	If $\psi(w)<0$, applying Lemma \ref{kq1a} with $\Gamma_1=\{w+t b: t\in (-\infty, 0]\}$ and $\Gamma_2=\{w+t c: t\in (-\infty, 0]\}$, one obtains the parabola $F(\curlyvee)$ and $\Gamma_1\cup\Gamma_2$ has at least one common point. Consequencely, $F(\mathcal{H})\cap (\{w\}-\bigwedge)\ne \emptyset$. Therefore, $w\in F(\mathcal{H})+\bigwedge$.
	
So, we have proved that for all cases, $w$ always belongs to $F(\mathcal{H})+\bigwedge$.
\end{proof} 

\begin{remark}\label{rm1}
In the proofs of Lemma \ref{kq1b} and Theorem \ref{kq2}, we rely solely on the properties of a quadratic mapping on a set and the affine nature of \(\mathcal{H}\), specifically that any line passing through two points in \(\mathcal{H}\) remains within \(\mathcal{H}\). Consequently, the conclusion remains valid if \(\mathcal{H}\) is replaced by an affine manifold \(\mathcal{M}\) contained in \(\mathcal{H}\).
\end{remark}

Applying Theorem \ref{kq2} and Remark \ref{rm1} along with a straightforward separation approach, we obtain the following result, which serves as the S-Lemma in Hilbert space -- an extension of the classical S-Lemma (see \cite{Yakubovich}, \cite{Terlaky}, and references therein).
\begin{corollary}Let $f(x)$ and $g(x)$ be quadratic functions on $\mathcal{M}$ $(\mathcal{M}=\mathcal{H}$ or an affine manifold contained in \(\mathcal{H})\). If there exists some \( x^* \in \mathcal{M} \) such that \( g(x^*) < 0 \), then the following statements are equivalent:  
	\begin{itemize}
		\item[(i)] $\not\exists x\in \mathcal{M}$ such that $f(x)<0$ and $g(x)\leq 0$.
		\item[(ii)] $\exists \lambda \geq 0$ such that $f(x)+\lambda g(x)\ge 0$ for all $x\in\mathcal{M}$.
	\end{itemize}
\end{corollary}

When \(\mathcal{H} = \mathbb{R}^n\) and $\bigwedge=\mathbb{R}^2_+$, let \( H \) be a real matrix of order \( m \times n \), and let \( d \in \mathbb{R}^n \) and \( x_0 \in \mathbb{R}^n \). Then the set \( C = H^{-1}(d) = x_0 + \operatorname{ker} H \) is an affine manifold contained in \(\mathcal{H}\). Relying on Theorem \ref{kq2} and Remark \ref{rm1}, we obtain a stronger version of the result, known as the Relaxed Dines theorem.  
\begin{corollary}[The Relaxed Dines theorem, \cite{Flores-Carcamo}] Let $f: \mathbb{R}^n \rightarrow \mathbb{R}, g: \mathbb{R}^n \rightarrow \mathbb{R}$ be any quadratic functions, and $C \doteq H^{-1}(d)=x_0+\operatorname{ker} H$. If $\mu=\inf\{f(x): g(x)\leq 0\} \in \mathbb{R}$, then cone \[\left(F(C)-\rho(1,0)+\mathbb{R}_{++}^2\right)\text{ is convex for all } \rho \leq \mu.\]
\end{corollary}
Finally, when \(\mathcal{H} = \mathbb{R}^n\), with \(\bigwedge = \mathbb{R}^2_+\) and \(\mathcal{M}\) as an affine manifold in \(\mathcal{H}\), we arrive at an elegant result established by Tuy et al., as stated in Corollary 10 of \cite{TT}. Likewise, for \(\mathcal{H} = \mathbb{R}^n\), a related and notable result was proven by Flores-Bazan et al., as stated in Theorem 4.19 of \cite{Flores-Opazo}.


\noindent Huu-Quang Nguyen\\
Department of Mathematics, School of Natural Sciences Education, Vinh University, Vinh, Nghe An, Vietnam\\
email: quangdhv@gmail.com

\end{document}